\documentclass[reqno,a4paper,11pt]{amsart}
\usepackage{hyperref}
\usepackage[top=2.7cm,left=2.8cm,right=2.8cm,bottom=2.6cm]{geometry}
\usepackage[usenames,dvipsnames]{xcolor}
\usepackage{mathtools}
\usepackage{todonotes}
\hypersetup{colorlinks=true,citecolor=NavyBlue,linkcolor=Brown,urlcolor=Orange}

\usepackage{amssymb, comment, paralist, xspace, graphicx, url, amscd, euscript, mathrsfs,stmaryrd,epic,eepic,color}

\usepackage{tikz-cd}
\usepackage[all]{xy}
\usepackage{amsthm}
\usepackage{amsmath}
\usepackage{amsfonts}
\usepackage{dsfont}
\usepackage[shortlabels]{enumitem}
\usepackage{tikz}
\usetikzlibrary{shapes,arrows,decorations.pathreplacing,backgrounds,positioning,fit,matrix}

1

\allowdisplaybreaks[1]
\numberwithin{equation}{section}

\newlist{enumeratea}{enumerate}{2}
\setlist[enumeratea]{label=\upshape(\alph*), leftmargin=0pt, labelsep=5pt,itemindent=2\parindent}
\newlist{enumeratei}{enumerate}{2}
\setlist[enumeratei]{label=\upshape(\roman*),leftmargin=0pt, labelsep=5pt,itemindent=2\parindent}

\newlist{enumerate1}{enumerate}{2}
\setlist[enumerate1]{label=\upshape(\roman*),leftmargin=0pt,labelsep=5pt,itemindent=2\parindent}

\setlist[itemize]{leftmargin=\parindent,itemindent=\parindent}
\newtheorem*{namedtheorem}{\theoremname}
\newcommand{\theoremname}{testing}

\newtheorem{theorem}{Theorem}[section]
\newtheorem{proposition}[theorem]{Proposition}
\newtheorem{proposition-definition}[theorem]{Proposition-Definition}

\newtheorem{lemma}[theorem]{Lemma}
\newtheorem{conjecture}[theorem]{Conjecture}

\theoremstyle{definition}
\newtheorem{definition}[theorem]{Definition}

\newtheorem{remark}[theorem]{Remark}

\theoremstyle{remark}

\newcommand\cA{\mathcal{A}}

\newcommand\cE{\mathcal{E}}
\newcommand\cF{\mathcal{F}}

\newcommand\cH{\mathcal{H}}

\newcommand\cK{\mathcal{K}}
\newcommand\cL{\mathcal{L}}
\newcommand\cM{\mathcal{M}}

\newcommand\cO{\mathcal{O}}
\newcommand\cP{\mathcal{P}}

\newcommand\NN{\mathbb{N}}

\newcommand\PP{\mathbb{P}}
\newcommand\QQ{\mathbb{Q}}

\newcommand\ZZ{\mathbb{Z}}

\newcommand\rD{\mathrm{D}}

\newcommand\cha{{\rm char}}
\newcommand\NS{{\rm NS}}

\newcommand\Hom{{\rm Hom}}
\newcommand\Ext{{\rm Ext}}
\newcommand\cris{{\rm cris}}

\newcommand\arr{\ifinner\to\else\longrightarrow\fi}

\newcommand\ext{\operatorname{Ext}}

\newcommand\Aut{\operatorname{Aut}}

\newcommand\spec{\operatorname{Spec}}

\newcommand{\Pic}{\operatorname{Pic}}

\newcommand{\proj}{\operatorname{Proj}}

\newcommand\rk{\operatorname{rk}}

\newcommand\id{\mathrm{id}}

\newcommand\Sym{\operatorname{Sym}}

\DeclareMathOperator{\sym}{Sym}

\DeclareMathOperator \alb{alb}

\DeclareMathOperator \ch{ch}

\DeclareMathOperator \Sp{\mathfrak{sp}}

\DeclareMathOperator \td{td}
\DeclareMathOperator \km{Km}

\DeclareMathOperator{\tr}{Tr}

\renewcommand{\ch}{{\rm ch}}
\newcommand{\h}{{\mathfrak h}}

\newcommand{\et}{\mathrm{\acute{e}t}}
\DeclareMathOperator{\BK}{BK}

\DeclareMathOperator{\GL}{GL}
\DeclareMathOperator{\Gr}{Gr}

\DeclareMathOperator{\supp}{Supp}
\DeclareMathOperator{\amp}{Amp}
\DeclareMathOperator{\Ql}{\mathbb{Q}_{\ell}}
\DeclareMathOperator{\Zl}{\mathbb{Z}_{\ell}}
\DeclareMathOperator{\Qp}{\mathbb{Q}_{p}}

\renewcommand{\1}{\mathop{\mathds{1}}\nolimits} 

\DeclareMathOperator{\bl}{Bl}
\newcommand{\xbar}[1]{\overline{#1}}

\begin{document}
\title[Supersingular O'Grady varieties of dimension six]{Supersingular O'Grady varieties of dimension six}

\author{Lie Fu}
\address{Universit\'e Claude Bernard Lyon 1\\
Institut Camille Jordan\\
43 Boulevard du 11 novembre 1918\\
69622 Cedex Villeurbanne\\
France}
\address{$\&$}
\address{Radboud University\\
IMAPP\\
Heyendaalseweg 135\\
6525 AJ, Nijmegen\\
Netherlands}
\email{fu@math.univ-lyon1.fr}

\author{Zhiyuan Li}
\address{Shanghai Center for Mathematical Science\\
Fudan University\\
2005 Songhu Road\\
20438 Shanghai, China}
\email{zhiyuan\_li@fudan.edu.cn}

\author{Haitao Zou}
\address{Shanghai Center for Mathematical Science\\
Fudan University\\
2005 Songhu Road\\
20438 Shanghai, China}
\email{htzou17@fudan.edu.cn}

\date{\today}
\thanks{2020 {\em Mathematics Subject Classification:} 	14J28, 14J42, 14D22, 14M20, 14C15, 14C25.}
\thanks{{\em Key words and phrases.} K3 surfaces, symplectic varieties, hyper-K\"ahler varieties, moduli spaces, supersingularity, unirationality, motives.}
\thanks{Lie Fu is supported by the Radboud Excellence Initiative program. Zhiyuan Li is supported by NSFC grants for  General Program (11771086), Key Program (11731004) and the Shu Guang Program (17SG01) of Shanghai Education Commission.}
\maketitle

\begin{abstract}
O'Grady constructed a 6-dimensional irreducible holomorphic symplectic variety by taking a crepant resolution of some moduli space of stable sheaves on an abelian surface. In this paper, we naturally extend O'Grady's construction to fields of positive characteristic $p\neq 2$, called OG6 varieties. Assuming $p\geq 3$, we show that a supersingular OG6 variety is unirational, its rational cohomology group is generated by algebraic classes, and its rational Chow motive is of Tate type. These results confirm in this case the generalized Artin--Shioda conjecture, the supersingular Tate conjecture and the supersingular Bloch conjecture proposed in our previous work, in analogy with the theory of supersingular K3 surfaces.
\end{abstract}

\section{Introduction}
The study on supersingular irreducible symplectic varieties was started in the previous work of the first two authors \cite{FL18}. As a continuation, we study in this paper O'Grady's 6-dimensional irreducible symplectic varieties in positive characteristics, and investigate the general conjectures formulated in \cite{FL18} for such varieties.

\subsection{Background}
In complex geometry, \textit{irreducible symplectic manifolds}, also known as \textit{compact hyper-K\"ahler manifolds}, are by definition simply connected compact K\"ahler manifolds admitting a nowhere degenerate holomorphic 2-form that is unique up to scalars. Together with complex tori and Calabi--Yau manifolds, they form the fundamental building blocks for compact K\"ahler manifolds with vanishing first Chern class, by the Beauville--Bogomolov decomposition theorem \cite{Bogomolov78, MR730926}. We refer the readers to \cite{MR1664696, MR1963559, OGradySurvey} for an overview of the study of these varieties.

By definition, 2-dimensional irreducible symplectic varieties are nothing but K3 surfaces. A significant part of the theory of complex K3 surfaces can be generalized to higher-dimensional irreducible symplectic varieties: deformation theory, quadratic form on the second cohomology group \cite{MR730926}, moduli spaces and period maps, global Torelli theorems \cite{VerbitskyTorelli, VerbitskyTorelliErrata, MarkmanSurvey, HuybrechtsTorelli}, and conjecturally their algebraic cycles \cite{BeauvilleVoisin04, Beau07, Voisin08}, etc. On the other hand, over a field of positive characteristic, the study of K3 surfaces has also been an active and fruitful research topic. One instance is the Tate conjecture \cite{Tate65} for K3 surfaces: the case of finite height was solved in \cite{MR723215, MR819555}, while a complete proof in the supersingular case was achieved only recently \cite{Ma14, Ch13, Pe15, Ch16, KM16}. It is a natural task to explore the higher-dimensional analogue of K3 surfaces in positive characteristics. We use the following working definition proposed in \cite{FL18} (see also \cite{YangISV19}):

\begin{definition}
A smooth projective variety $X$ over an algebraically closed field $k$ is called \emph{symplectic} if 
\begin{itemize}
	\item $\pi_1^{\operatorname{\et}}(X)=0$;
	\item $X$ admits a nowhere degenerate closed algebraic 2-form.
\end{itemize}
In particular, $X$ is even dimensional and has trivial canonical bundle.
A symplectic variety $X$ is called \textit{irreducible symplectic} if moreover we have $\dim H^0(X,\Omega_X^2)=1$. 
\end{definition}

An important feature that only appears in positive characteristic is the existence of supersingular varieties. Fix an algebraically closed field $k$ of characteristic $p>0$. Recall that a K3 surface $S$ defined over $k$ is called supersingular, following Artin \cite{Ar74}, if the second crystalline cohomology $H^2_\cris(S/W(k))$ is a supersingular $F$-crystal, that is, its Newton polygon is a straight line. The aforementioned Tate conjecture implies that it is equivalent to Shioda's condition \cite{Sh74} that the second ($\ell$-adic or crystalline) cohomology group is generated by divisor classes. A systematic study of supersingular  symplectic varieties was initiated in our previous work \cite{FL18}, where two types of supersingularity were defined. More precisely, a  symplectic variety $X$ is called
\begin{itemize}
    \item $2^{nd}$-\textit{Artin supersingular}, if its second crystalline cohomology group $H^2_{\cris}(X/W)$ is a supersingular $F$-crystal;
    \item $2^{nd}$-\textit{Shioda supersingular}, if its Picard number is maximal: $\rho_X=b_2(X)$.
\end{itemize}
Similarly, there are stronger notions of supersingularity that concern the whole cohomology group. For simplicity, let us assume in addition that all the odd Betti numbers of $X$ are zero. Then a symplectic variety $X$ is called
\begin{itemize}
    \item \textit{fully Artin supersingular}, if $H^{2i}_\cris(X/W)$ is a supersingular $F$-crystal for any $i\in \mathbb{N}$;
    \item \textit{fully Shioda supersingular}, if all $\ell$-adic and crystalline cohomology groups are rationally generated by algebraic classes.
\end{itemize}
It is easy to see that the $2^{nd}$ (resp.~fully) Shioda supersingularity implies the $2^{nd}$ (resp.~fully) Artin supersingularity. See \cite[Definitions 2.1 and 2.3]{FL18}) for more details.

\subsection{Conjectural picture}
Motivated by properties of supersingular K3 surfaces, we proposed the following conjecture in \cite{FL18}. It relates the cohomological, geometric, and motivic aspects of supersingular symplectic varieties.

\begin{conjecture}
\label{conj:main}
Let $X$ be a symplectic variety defined over an algebraically closed field $k$, with vanishing odd Betti numbers. If $X$ is  $2^{nd}$-Artin supersingular, then 
\begin{enumeratei}
\item $X$ is fully Shioda supersingular, i.e.~the ($\ell$-adic and crystalline) cycle class map is surjective.
\item $X$ is unirational;
\item The rational Chow motive is of Tate type. In particular, the ($\ell$-adic and crystalline) cycle class map is injective.
\end{enumeratei}
\end{conjecture}

\begin{remark} In the above conjecture,
\begin{itemize}
    \item the assumption on Betti numbers is for the simplicity of the statement. See \cite[Conjectures A and B]{FL18} for the general version where non-zero odd cohomology is allowed;
    \item Part (i) says that all notions of supersingularity defined in \cite{FL18} are equivalent. This can be seen as a generalization of the Tate conjecture \cite{Tate65} in the supersingular case.
    \item Part (ii) is the generalization of the Artin--Shioda conjecture on the unirationality of supersingular K3 surfaces \cite{Ar74, Shioda77}. 
    \item Part (iii) can be viewed as a Bloch-type conjecture (\emph{cf.}~\cite{MR2723320}, \cite[\S~11.2]{MR2115000} and \cite[Conjecture 23.21]{MR1988456}) in the supersingular case.
\end{itemize}
\end{remark}

\noindent Here are some evidence of Conjecture \ref{conj:main}:
\begin{itemize}
    \item For K3 surfaces, 
    \begin{enumeratei}
    \item is the Tate conjecture, proved in \cite{Ma14, Ch13, Pe15, Ch16, KM16}.
    \item was proved for $p=2$ by Rudakov--\v Safarevi\v c \cite{RS78} and remains open in general (see however \cite{Li15, BL18, BL19}).
    \item was proved by Fakhruddin \cite{Fak02} (his result is for $p\geq 5$, but one can show it in general by using \cite{BL18}, see \cite[\S 4.5]{FL18}). 
    \end{enumeratei} 
    \item For smooth projective moduli spaces of stable sheaves on K3 surfaces, \cite[Theorem 1.3]{FL18} established (i) and (iii), and reduces (ii) to the unirationality of supersingular K3 surfaces.
    \item For the Albanese fiber of a moduli space of stable sheaves on an abelian surface, provided that it is smooth, \cite[Theorem 1.4]{FL18} proved the analogue of (i) and (iii) with non-vanishing odd cohomology, and showed that it is rationally chain connected, a weaker version of (ii).
\end{itemize}

Note that in these results, the symplectic varieties considered are all deformation equivalent to one of the examples constructed by  Beauville \cite{MR730926}, namely, Hilbert schemes of K3 surfaces and generalized Kummer varieties. The primary goal of this paper is to seek for examples of different deformation type that satisfy Conjecture \ref{conj:main}.

\subsection{Main results}
Over the complex numbers, apart from the two infinite series of deformation types of irreducible symplectic varieties of Beauville, so far we only know other two sporadic families discovered by O'Grady \cite{OG10, OG6}:
\begin{itemize}
    \item Let $S$ be a projective K3 or abelian surface. Let $v_0\in H^*(S, \ZZ)$ with $v_0^2=2$ and $H$ a sufficiently general polarization. The moduli space $M_H(S, 2v_0)$ of $H$-semistable sheaves on $S$ with Mukai vector $2v_0$ is singular. In O'Grady \cite{OG10}, for $v_0=(1, 0, -1)$, a 10-dimensional holomorphic symplectic manifold $\widetilde{M}_H(S, 2v_0)$ was constructed as a crepant resolution of $M_H(S, 2v_0)$. This construction was extended to any $v_0$ with $v_0^2=2$ by Lehn--Sorger \cite{LS06}, obtaining symplectic varieties that are deformation equivalent to O'Grady's example by \cite{PR13}. When $S$ is a K3 surface, the crepant resolution $\widetilde{M}_H(S, 2v_0)$ is irreducible symplectic, called an \textit{OG10 variety}.
    \item When $S$ is an abelian surface, for any $v_0$ with $v_0^2=2$, the Albanese map of the crepant resolution 
    \[\operatorname{alb} \colon \widetilde{M}_H(S, 2v_0)\longrightarrow S\times S^\vee\]
    is an isotrivial fibration. Any fiber of this fibration is a 6-dimensional irreducible symplectic variety \cite{OG6, LS06}, called an \textit{OG6 variety}. 
\end{itemize}
We will adapt the construction of O'Grady and Lehn--Sorger to positive characteristics $>2$ (Theorem \ref{thm:OGconstruction} and Theorem \ref{prop:cohomological}). The resulting variety is still called an OG6 variety. The following is the main result of this article, confirming Conjecture~\ref{conj:main} for supersingular OG6 varieties.

\begin{theorem} 
\label{thm:main}
Let $k$ be an algebraically closed field of characteristic $p\geq 3$, and $A$ an abelian surface defined over $k$. Let $v_0$ be a Mukai vector on $A$ with $v_0^2=2$ and set  $v=2v_0$. For any $v$-generic polarization $H$, let $K_H(A, v)$ be the Albanese fiber of the moduli space of $H$-semistable sheaves with Mukai vector $v$. Then O'Grady's crepant resolution $\widetilde{K}_H(A,v)$ is a symplectic variety and the following conditions are equivalent:
\begin{enumerate1}
    \item  $A$ is supersingular;
   \item $\widetilde{K}_H(A,v)$ is $2^{nd}$-Artin supersingular;
   \item $\widetilde{K}_H(A,v)$ is unirational;
   \item The rational Chow motive of $\widetilde{K}_H(A,v)$ is of Tate type;
   \item $\widetilde{K}_H(A,v)$ is fully Shioda supersingular: all rational ($\ell$-adic or crystalline) cohomology groups are generated by algebraic classes.
\end{enumerate1}
Moreover, if one of these conditions is satisfied, $\widetilde{K}_H(A,v)$ is an irreducible symplectic variety. 
\end{theorem}


\begin{remark}
According to Theorem \ref{thm:main}, the OG6 construction provides a 1-dimensional family of supersingular symplectic varieties. It is very interesting to investigate the supersingular locus in the moduli space of symplectic varieties of OG6 deformation type. A natural question is whether every supersingular symplectic varieties of OG6 type is birational to an OG6 construction $\widetilde{K}_H(A,v)$. We believe that similarly to the case of symplectic varieties of $K3^{[n]}$ deformation type, there exist supersingular symplectic varieties of OG6 deformation type arising as moduli spaces of \textit{twisted} sheaves on  supersingular abelian surfaces, which are not birational to $\widetilde{K}_H(A,v)$. In particular, we expect that the supersingular locus has dimension at least $2$. 
\end{remark}

The proof of Theorem \ref{thm:main} essentially splits into two steps:
\begin{enumeratea}
\item Upon improving the techniques developed in \cite{FL18}, we show in Theorem \ref{thm:Birational} that given a supersingular abelian surface $A$, all such OG6 varieties $\widetilde{K}_H(A,v)$, for different choices of $v$, are birationally equivalent with good lifting properties (Definition \ref{def:quasiliftable}). 
\item For certain choice of Mukai vector $v_0$, we use a construction of Mongardi--Rapagnetta--Sacc\`a \cite{MRS18}, which relates the OG6 variety to a moduli space of stable sheaves on the Kummer K3 surface of $A$, which was thoroughly studied in our previous work \cite{FL18}.
\end{enumeratea}
The technique of lifting to characteristic zero is used in several places, but note that the proof of the unirationality result does not rely on the lifting argument.

The paper is organized as follows. In \S \ref{chapter:OGLS} we generalize the construction of O'Grady \cite{OG6} and Lehn--Sorger \cite{LS06} to odd characteristic and show that thus obtained OG6 varieties are symplectic varieties. In \S \ref{chapter:BiratSSOG6}, we establish the birational equivalences between supersingular OG6 varieties (Step (a) above). In \S \ref{chapter:Proof}, we adapt Mongardi--Rapagnetta--Sacc\`a's construction \cite{MRS18} and finish the proof of Theorem \ref{thm:main} (Step (b)).
 
\noindent\textbf{Convention:} 
\begin{itemize}
    \item $k$ is the base field, which is always assumed to be algebraically closed with $\cha(k)=p >0$.
    \item $W=W(k)$ is the ring of Witt vectors of $k$, which comes  equipped with a morphism $\sigma:W\rightarrow W$ induced by the Frobenius morphism.
    \item $K$ is the fraction field of $W$.
    \item $B_{\cris}$ is the crystalline period ring of $K$.
    \item For an algebraic variety $X$ over a field $E$ and an integer $i$, $H^i_{\ell}(X)$ denotes the $i$-th integral $\ell$-adic cohomology $H^i_{\et}(X_{\xbar{E}},\Zl)$, and $H^i_{\ell}(X)_{\Ql}$ denotes the rational $\ell$-adic cohomology if $\ell \neq \cha(E)$.
    \item For a $k$-algebraic variety $X$ and $i\in \mathbb{N}$, we denote by $H^i_{\cris}(X/W)$ the $i$-th integral crystalline cohomology of $X$, which is a $W$-module whose rank is equal to the $i$-th Betti number of $X$.
   We set $H^i (X):=H^i_{\cris}(X/W)/\text{torsion}$ and $~H^i(X)_K=H^i_{\cris}(X/W)\otimes_{W} K$. 
 Then $H^i(X)$ is endowed  with a natural $\sigma$-linear map $\varphi: H^i(X)\rightarrow H^i(X)$ induced from the absolute Frobenius morphism $F\colon X\to X$. The pair $(H^i(X), \varphi)$ (\emph{resp.}~$(H^i(X)_K, \varphi_{K})$) forms therefore an $F$-crystal (\emph{resp.}~$F$-isocrystal).
    \item For a Breuil--Kisin module $M$ and an integer $n$, the notation $M\lbrace n \rbrace$ stands for the $n$-th Tate twist defined in \cite[Example 4.2]{BMS18}.
    \item $\h$ is the contravariant functor that associates a variety its rational Chow motive.
\end{itemize}

\noindent\textbf{Acknowledgement:} We thank Giovanni Mongardi for helpful discussions. We thank also the referee for his or her stimulating comments.

\section{O'Grady--Lehn--Sorger's construction in positive characteristic}\label{chapter:OGLS}
In this section, we adapt the construction of Lehn--Sorger and  O'Grady to obtain a 6-dimensional symplectic variety over a field of positive characteristic, as a crepant resolution of a singular moduli spaces of semistable sheaves on an abelian surface.

\subsection{Moduli spaces of sheaves on abelian surfaces}
One natural way to construct symplectic varieties is to consider certain moduli spaces of coherent sheaves on abelian and K3 surfaces. Over the complex numbers, this is an observation of Mukai \cite{Mukai84}, which is adapted to fields of positive characteristic in \cite[Propositions 4.4 and 6.9]{FL18}. Let $A$ be an abelian surface over $k$.  We denote by  $$\widetilde \NS(A)=\ZZ\cdot {\mathds 1}\oplus \NS(A)\oplus \ZZ\cdot\omega$$ the \emph{algebraic Mukai lattice} of $A$, where $\mathds 1$ is the fundamental class of $S$ and $\omega$ is the class of a point. An element $r\cdot {\mathds 1}+L+s\cdot \omega$ of $\widetilde \NS(A)$, where $r, s\in \ZZ$ and $L\in \NS(A)$, is often denoted by $(r, L, s)$. The lattice structure on $\widetilde \NS(A)$ is given by the \emph{Mukai pairing} $\left<-,-\right>$: 
\begin{equation}\label{pairing}
\left<(r, L, s), (r', L', s')\right>=(L\cdot L')-rs'-r's \in \ZZ.
\end{equation}
For a coherent sheaf $\cF$ on $A$, its \emph{Mukai vector} is the same as its Chern character:
$$v(\cF):=\ch(\cF)\sqrt{\td(A)}= \left(\rk(\cF), c_1(\cF), \chi(\cF) \right)\in \widetilde{\NS}(A). $$

Let us recall the notion of genericity of polarization. 

\begin{definition}
\label{def:generic}
Let $v\in \widetilde \NS(A)$.
An ample line bundle $H$ is called $v$-\textit{generic} if
\begin{itemize}
    \item when $v$ is primitive, all $H$-semistable sheaves on $A$ are  $H$-stable;
    \item when $v= n v_0$ with $n\in \NN$ and $v_0\in \widetilde \NS(A)$ primitive,   $H$ is $v_0$-generic and for any stable factor $\cE$ of a $H$-semistable sheaf $\cF$, we have $v(E)= m v_0$ with $1\leq m \leq n$.
\end{itemize}
\end{definition}

Let us discuss the wall-and-chamber structure on the ample cone $\text{Amp}(A)$ with respect to $v$. This plays an important role in the proof of our main theorem. 
\begin{lemma}
For any torsion free $H$-semistable sheaf $\cF$ on $A$ over an  algebraically closed field $k$, the Bolgomolov inequality
\begin{equation}\label{eq:Bogomolovinequality}
\Delta(\cF) =  2 \rk(\cF) c_2(\cF) - (\rk(\cF) -1) c_1^2(\cF) \geq 0 
\end{equation}
holds.
\end{lemma}
\begin{proof}
We will show that $\cF$ is strongly semistable in Lemma \ref{lem:ss=strongss}, thus we can invoke Langer's result \cite[Theorem 0.1]{La04}.
\end{proof}

Given a Mukai vector $v \in \widetilde{\NS}(A)$, we can see the $\Delta(-)$ is constant for all pure sheaves whose Mukai vectors are $v$. Thus we will denote it by $\Delta$ for simplicity if the Mukai vector is fixed. If $\rk(v) >0$ and $\Delta > 0$, then we define the following set
\[
W_{v}:=\left\{ D \in \NS(A) \Big| - \frac{\rk(v)}{4} \Delta \leq D^2 <0 \right\}.
\]
If $\Delta=0$, then we will set $W_v = \emptyset$.
\begin{definition}
The $v$-wall in $\amp(A)$ associated to some $D \in W_v$ is the hyperplane in $\amp(A)$ defined as
\[
W^D \coloneqq \left\{ H \in \amp(A) \big| H \cdot D =0 \right\}.
\]
A $v$-chamber means a connected component of
\[
\amp(A) \backslash \bigcup_{D \in W_v}W^D.
\]
\end{definition}
\begin{lemma}\label{prop:10-1wallstructure} Let $H$ be an ample divisor on $A$. If for any $(-2)$-class $D \in \NS(A)$ we have $D \cdot  H \neq 0$, then $H$ is $(2,0,-2)$-generic. In particular, if $H$ lies in any $(2,0,-2)$-chamber, then $H$ is $(2,0,-2)$-generic.
\end{lemma}
\begin{proof}
It is sufficient to show that the stable factors of an $H$-polystable sheaf $\cF$ with $v(\cF)=(2,0,-2)$ are with Mukai vector $(1,c_1, -1)$ such that $c_1=0$ or $c_1$ is a $(-2)$-class in $\NS(A)$.

The purity of $\cF$ implies that its stable factors are of rank 1. Thus $\cF$ can be written as
\[
\cF = \cE_1 \oplus \cE_2
\]
such that $v(\cE_1) =(1, c_1, s)$ and $v(\cE_2) = (1, c_1', s')$. They satisfy the following equations:
\begin{equation}\label{eq:10-1equation}
\begin{aligned}
c_1 + c_1' = 0, & \quad c_1 \cdot H = c_1' \cdot H =0, \\
& s +s' =2.
\end{aligned}
\end{equation}
From these equations, we can deduce that $c_1'^2=c_1^2 = c_2(\cE_1) + c_2(\cE_2) -2$. Thus $c_1'^2 = c_1^2 \geq -2$ as the Bogomolov's inequality \eqref{eq:Bogomolovinequality}  implies that $c_2(\cE_i) \geq 0$, . If $c_1^2 \geq 0$, then $c_1=0$ follows from the Hodge index theorem, as $c_1 \cdot H =0$. Hence the set of solutions of $c_1$ to the \eqref{eq:10-1equation} is contained in
\[
\left\{ c_1 \in \NS(A) \big| c_1 \cdot H=0, c_1^2=-2 \right\} \cup \{0\}.
\]
If $\left\{ c_1 \in \NS(A) \big| c_1 \cdot H=0, c_1^2=-2 \right\} = \emptyset$, then $H$ is $(2,0,-2)$-generic, which is our claim.
\end{proof}

\begin{remark}
For a general $v$ such that $\rk(v) >0$, it is still true that if an ample divisor $H$ lies in a $v$-chamber, then $H$ is $v$-generic.
\end{remark}

The following useful lemma shows that the $v$-genericity is preserved under lifting.
\begin{lemma}\label{lemma:genericitylifting}
Let $v_0=(r, c_1, s)$ be a Mukai vector. 
Let $(\cA, \cH, \widetilde{c}_1)$ be a lifting of the triple $(A,H,c_1)$ to a finite extension $W'$ of $W$. Let $K'$ be the fraction field of $W'$. Then $\cH_{K'}$ is $(2r,2\widetilde{c}_1,2s)$-generic if $H$ is $2v_0$-generic.
\end{lemma}
\begin{proof}
We can see $\cH_{K'}$ is $(r,\widetilde{c}_1,s)$-generic from the properness of the relative moduli space $\cM_{\cH}(\cA, (r,\widetilde{c}_1, s))$.

Since the semistability is stable under field extensions, it is sufficient to consider the $\cH_{\xbar{K}}$-semistable sheaves on the geometric generic fiber $\cA_{\xbar{\eta}}$. Suppose that $E_{\bar{K}}$ is $\cH \otimes \bar{K}$-polystable sheaves on the geometric generic fiber $\cA_{\bar{\eta}}$ such that $v(E_{\bar{K}})= 2v_0 \otimes \bar{K}$. If $F_{\bar{K}}$ a stable factor of $E_{\bar{K}}$, then we can find a stable factor $F_k$ of $E_k$ such that $v(F_k)= v(F_W) \otimes k$, since the specialization $\NS(\cA_{\bar{\eta}}) \to \NS(A)$ is injective, by \cite[Proposition 3.6]{Maulik2012} and the fact that N\'eron--Severi groups of abelian surfaces are torsion-free. Thus the stable factors of $E_{\bar{K}}$ are all with $v_0 \otimes \bar{K}$ by the assumption that $H$ is $2v_0$-generic.
\end{proof}

\subsection{O'Grady--Lehn--Sorger's construction: positive characteristics and lifting}
\label{sec:OGLS-Lifting}
Keep the notations as before. 
Given an element $v=(r,c_1,s)\in \widetilde \NS(A)$ such that $r\geq 0$, $(r,s)\neq (0,0)$ and $\left<v,v\right>\geq 0$,  together with a $v$-generic ample line bundle $H$, we consider the moduli space  $M_H(A,v)$ of Gieseker--Maruyama $H$-semistable sheaves on $S$ with Mukai vector $v$. 
Fix a point $[\cF_0]\in M_H(A, v)$, then the following morphism using the determinant and the second Chern class
\begin{align*}
     M_{H}(A,v) &\to \Pic^{0}(A)\times A\\
    \cF &\mapsto \big(c_1(\cF)-c_1(\cF_0), \operatorname{alb}(c_2(\cF)-c_2(\cF_0))\big),
\end{align*}
is the Albanese morphism for $M_{H}(A,v)$ and it is an isotrivial fibration. It is shown in \cite[Proposition 6.8]{FL18} that when $v$ is primitive with $p\nmid v^2$ and $H$ is $v$-generic, the Albanese fiber, denoted by $K_{H}(A,v)$, is a symplectic variety of generalized Kummer type of dimension $\left<v, v\right>-2$, provided that it is smooth.

If the Mukai vector $v\in \widetilde{\NS}(A)$ is not primitive, then the moduli space of $H$-semistable sheaves $M_H(A,v)$ is in general singular even for a $v$-generic polarization $H$. We will concentrate in the particularly interesting case where $v=2v_0$ for some primitive Mukai vector $v_0$ with $v_0^2=2$. Given a $v$-generic polarization $H$ of the abelian surface $A$, there is a stratification of $M_H(A,v)$ by the singular loci:
\begin{equation}
\label{eq:blow-up1}
\Omega_v \subset \Sigma_{v} \subset  M_H(A,v),
\end{equation}
where
\begin{itemize}

\item The closed variety $\Sigma_v$ is the singular locus of $M_H(A,v)$ consisting of strictly semistable sheaves (or equivalently, direct sums of two stable sheaves with Muakai vector $v_0$).
\item  The closed variety $\Omega_v$ is the singular locus of $\Sigma_v$ which  parameterizes sheaves of the form $F^{\oplus 2}$ with $F$ a stable sheaf with  Mukai vector $v_0$. 
\end{itemize}
It is clear that $\Sigma_v$ is isomorphic to the symmetric square of $\Omega_v$.

Over the complex numbers, O'Grady \cite{OG10, OG6} (for special $v_0$ and later Rapagnetta \cite{Rap07} for all $v_0$ with $v_0^2=2$) constructed a symplectic resolution $\widetilde{M}_H(A,v)$ of $M_H(A,v)$ in three steps:
\begin{enumerate1}
    \item Blow up $M_H(A,v)$ along $\Omega_v$ to get a space denoted $\overline{M}$. Let $\overline{\Omega}_v$ be the exceptional divisor and let $\overline{\Sigma}_v$ be the strict transform of $\Sigma_v$.
    \item Blow up $\overline{M}$ along $\overline{\Sigma}_v$, resulting a space denoted $\widehat{M}$. Let $\widehat{\Omega}_v$ be the strict transform of $\overline{\Omega}_v$, which is isomorphic to the Hilbert square of $\Omega_v$ (hence smooth).
    \item Perform an extremal contraction on $\widehat{M}$, which contracts $\widehat{\Omega}_v$ as a $\mathbb{P}^2$-bundle onto $\widetilde{\Omega}_v$. The resulting variety is denoted by $\widetilde{M}_H(A,v)$, which is shown to be a symplectic resolution of $M_H(A,v)$, and $\widetilde{\Omega}_v$ is a 3-dimensional quadric bundle over $\Omega_v$.
\end{enumerate1}

In \cite{LS06}, Lehn and Sorger proved that O'Grady's construction for $\widetilde{M}_H(A,v)$ is isomorphic to the blow-up of $M_H(A,v)$ along the singular locus $\Sigma_v$, and $\widehat{M}$ is obtained by a further blow-up along $\widetilde{\Omega}_v$, the preimage of $\Omega_v$. 

To summarize, we have the following diagram:
\begin{equation}\label{diag:OGLS}
\begin{tikzcd}
 & \widehat{M}=\bl_{\overline{\Sigma}_v}\!\bl_{\Omega_v}\!M_H(A,v) \ar[rd,"\text{blowing up $\overline{\Sigma}_v$}"] \ar[ld,"\text{blowing up $\widetilde{\Omega}_v$}","\text{contracting $\widehat{\Omega}_v$}"'] & \\
 \widetilde{M}_H(A,v)=\bl_{\Sigma_v}M_H(A,v) \ar[rd,"\text{crepant}","\text{blowing up $\Sigma_v$}"'] && \overline{M}=\bl_{\Omega_v}\!M_H(A,v) \ar[ld,"\text{blowing up $\Omega_v$}"] \\
 & M_H(A,v)&
\end{tikzcd}
\end{equation}
See \cite{kaledin06} for a thorough discussion for more general Mukai vectors.

Return to a field $k=\bar{k}$ with characteristic $p>0$. The main goal of this section is to show that everything works if $p\neq 2$, that the whole diagram \eqref{diag:OGLS} actually lifts to characteristic zero (Proposition \ref{prop:liftability}), and deduce information on the Albanese fiber, which will be the characteristic-$p$ version of O'Grady's variety. Before start, we remark that the stratification by singular loci \eqref{eq:blow-up1} clearly remains valid. We first adapt Lehn--Sorger's construction in odd characteristics.

\begin{theorem}
\label{thm:OGconstruction}
Let $k$ be an algebraically closed field such that $\cha(k)=p \neq 2$. Let $v=2v_0$ be a Mukai vector such that $v_0\in \widetilde{\NS}(A)$ is primitive and $v_0^2=2$. Then the blow-up of $M_H(A,v)$ along $\Sigma_v$ gives a symplectic resolution $\widetilde{M}_{H}(A,v)\rightarrow M_H(A,v)$. Moreover, the preimage of $\Omega_v$ is a 3-dimensional quadric bundle over $\Omega_v$.
\end{theorem}
\begin{proof} 
We follow the same idea as in \cite{LS06} and \cite{kaledin06}, paying special attention to the singular loci of $M_v\coloneqq M_H(A,v)$ in positive characteristic and the failure of the Kodaira vanishing theorem.
Denote the blow-up
\begin{equation}
\label{eq:Blow-up}
   \pi_v \colon  \bl_{\Sigma_v}M_H(A,v) \to M_H(A,v).
 \end{equation}

Firstly we shall give a local description around points in $\Sigma_v$. 
 For any $[\cE]=[\cF_1\oplus \cF_2] \in \Sigma_v$, we will naturally identify $\Ext^1(\cE,\cE)$ with its dual by Serre duality. Let $k[\Ext^1(\cE,\cE)]^\wedge$ be the completion of the polynomial ring $$k[\Ext^1(\cE,\cE)]\coloneqq \sym^\bullet(\ext^1(\cE,\cE)^*)$$ at the maximal ideal of its origin. The \emph{Kuranishi map} is  encoded in a formal function:
\[
\kappa \colon \Ext^2(\cE,\cE)_0^* \to k[\Ext^1(\cE,\cE)]^\wedge,
\]
and we can write  $\kappa$ in the form $\kappa = \sum_{i=2}^{\infty} \kappa_i$, satisfying the following properties:
\begin{itemize}
    \item $\kappa$ is $\Aut(\cE)$-equivariant,
    \item any element in the image of $\kappa$, viewed as a formal function on $\Ext^1(\cE,\cE)$, vanishes on $\Ext^1(\cF_i, \cF_i)$ for any $i$ (this relies on the assumption that $\cha(k) \neq 2$), 
    \item the quadratic term of $\kappa$, viewed as a map $\kappa_2\colon \Ext^1(\cE,\cE)\to \Ext^2(\cE,\cE)_0$, is given by $e\mapsto \frac{1}{2} e\smile e$, where $\smile$ is the Yoneda product,
    \item Let $J$ be the ideal generated by the image of
    $\kappa$. Then we have an isomorphism
    \begin{equation}\label{eq:Kuranishilocaliso}
        \widehat{\mathcal{O}}_{M_H(A,v),[\cE]} \cong (k[\Ext^1(\cE,\cE)]^\wedge /J)^{\Aut(\cE)}.
    \end{equation}
\end{itemize}
An explicit construction of $\kappa$ can be found in \cite[Appendix A]{LS06}, which works in any characteristic $p\neq 2$. 

(i) If $[\cE]=[\cF_1\oplus \cF_2] \in \Sigma_v\backslash \Omega_v$ with $[\cF_1] \neq [\cF_2]$, we can write $x \in \Ext^1(\cE,\cE)$ as $\sum_{i,j=1}^2 x_{ij}$, such that $x_{ij} \in \Ext^1(\cF_i,\cF_j)$. Since any element in the image of $\kappa$ vanishes on $\Ext^1(\cF_1,\cF_1) \oplus \Ext^1(\cF_2,\cF_2)$, we have
\[
\kappa_2(x) = x_{12} \smile x_{21}.
\]
 Let $\mathscr{Q}^{ss}$ be the parameter space of $H$-semistable pure sheaves on $A$ with the Mukai vector $v$. An analogous proof as in \cite[Proposition 1.4.1]{OG10} shows that,  the normal cone $C_{\Sigma_v \backslash \Omega_v} \mathscr{Q}^{ss}$  at the point $[\cE]$ is isomorphic to 
 $ \kappa_2^{-1}(0)$:
 \[
C \coloneqq (C_{\Sigma_v \backslash \Omega_v} \mathscr{Q}^{ss})_{[\cE]} \cong \left\{ (x_{12}, x_{21}) \in \Ext^1(\cF_1,\cF_2) \oplus \Ext^1(\cF_2,\cF_1)| x_{12} \smile x_{21}=0 \right\},
\]
where $C$ is of maximal rank. In addition, the action of $\Aut(\cE)=\mathbb{G}_m$ on $\kappa_2^{-1}(0)$ is given by
\[
\lambda(x_{12},x_{21})= (\lambda x_{12},\lambda^{-1} x_{21}).
\]
Thus, we can see that the point $[\cE] \in \Sigma_v \backslash \Omega_v \subset M_H(A,v)$ is of $A_1$-singularity along $\Sigma_v$.

(ii)  If $[\cE]=[\cF^{\oplus 2}] \in \Omega_v$,  a direct computation shows that \[\Aut(\cE) \cong \GL_2, ~\Ext^1(\cE,\cE)\cong \mathfrak{gl}_2 \otimes V~ \hbox{and} ~\Ext^2(\cE,\cE)_0\cong \mathfrak{sl}_2\] where $V\coloneqq \Ext^1(\cF,\cF)$ is a 4-dimensional $k$-vector space. Let $\{v_1, v_2,v_3,v_4 \}$ be a basis of $V$ such that the symplectic form $\omega$ on $V$ is of the form
\begin{equation}
\label{eq:sympform}
T=\begin{pmatrix}
0& 1& &  \\
-1& 0& &  \\
 &  & 0& 1 \\
 &  & -1& 0
\end{pmatrix}
\end{equation}
which gives an identification $\mathfrak{gl}_2 \otimes V = \mathfrak{gl}^{\oplus 4}_2$. We may denote by $A_i$  the $i$-th projection $\mathfrak{gl}^{\oplus 4}_2 \to \mathfrak{gl}_2 \otimes \{v_i\}$  under this identification.  The (formal) local model of $(M_H(A, v),\mathcal{E})$ is described as follows:  let $R= \sym^{\bullet}(\mathfrak{sp}_4)$ and $I_0$ the ideal of $R$ generated by the coefficients of matrices $B^2$ with $B \in \mathfrak{sp}_4$. Let $Z \subseteq \Sp_4$ be the subvariety $\left\{B \in \Sp_4\big| B^2=0 \right\}$. Following \cite{LS06},  there is an isomorphism $Z\cong\spec(R/I_0)$ and we have
\[
    \widehat{\mathcal{O}}_{M_H(A,v),[\cE]}\cong k[[x_1,x_2,x_3,x_4]] \otimes_k R^{\wedge}/{I_0}R^{\wedge} \cong \widehat{\cO}_{\mathbb{A}^4_k \times Z,0}.
\]
Here,  $x_i \coloneqq \tr(A_i)$. The construction of the first isomorphism fails when $p=2$ as it relies on the classical invariant theory of $\text{SL}_2$, but valid when $p\neq2$; see \cite[Corollary 4.1]{DKZ2002}.

The singular locus $Z_s \subseteq Z$ is the set
\[
Z_s = \left\{ B \in Z\big| \rk(B) \leq 1\right\},
\]
whose definition ideal in $R$ is generated by the coefficients in $\Lambda^2B$, denoted by $L_0$. Formally locally, $Z_s \backslash 0$ is isomorphic to $\Sigma_v \backslash \Omega_v$  near $\mathcal{E}$. We want to show that blowing up $Z_s$ gives a resolution of $Z$.

Let $\Gr^{\omega}(2,V) \subset \Gr(2,V)$ be the Grassmannian variety of 2-dimensional Lagrangian subspaces of $V$. Define $\widetilde{Z} \subset Z \times \Gr^{\omega}(2,V)$ as
\[
\widetilde{Z} = \left\{ (B,U) \big| B(U) =0 \right\}.
\]
  Then the projection  $\pi: \widetilde{Z}\to Z$ is a smooth resolution because the  other projection $\tilde{Z}\to \Gr^{\omega}(2,V)$ can be identified with the projection of cotangent bundle $T^*\!\Gr^{\omega}(2,V) \to \Gr^{\omega}(2,V)$, which is smooth. Moreover, we see that the fibers of $\pi$ are 3-dimensional quadrics.

  We claim that the projection $\pi$ is isomorphic the blow-up of $Z$ along the singular locus $Z_s$, i.e.\ 
\[
\widetilde{Z}= \proj\left(\bigoplus_{n \geq 0}\pi_*(p^*\cO_{\mathbb{P}(\wedge^2)) V^*)}(2n))\right) \cong \proj\left( \bigoplus_{n \geq 0} L_0^n \right).
\]
It is sufficient to show that there is a surjection
\[
L_0^n \twoheadrightarrow \pi_*(p^*\cO_{\mathbb{P}(\wedge^2)) V^*)}(2n))
\]
as $\widetilde{Z}$ is irreducible and $\dim \widetilde{Z} = \dim \proj\left( \bigoplus_{n \geq 0} L_0^n \right)$. Lehn and Sorger's proof in \cite[Theorem 2.1]{LS06} was for characteristic 0, while everything also works in positive characteristics except the vanishing condition: 
\[H^1\left(\Gr^{\omega}(2,V), \Sym^k(T_{\Gr^{\omega}(2,V)})(1)\right)=0 \quad \text{for } k \geq 0,\]
which is used in \cite[Lemma 2.2]{LS06} as a corollary of the Griffiths vanishing theorem\footnote{Griffiths vanishing theorem says that for any ample sheaf $E$ on a smooth projective variety $X$, for $i \geq 1$ and $k \geq 0$, $H^i(X,\omega_X \otimes \Sym^k(E) \otimes \det(E))=0$.}. However, as $\Gr^{\omega}(2,V)$ is liftable, the Griffiths vanishing theorem holds (as a special case of Kodaira Vanishing Theorem) for $\cha(k)=p \geq \dim \Gr^{\omega}(2,V)=3$ by Delgine--Illusie; see \cite{DI87}. Now we can conclude that the blow-up $\pi_{v}$ in \eqref{eq:Blow-up} is a symplectic resolution of $M_H(A,v)$.
\end{proof}

Our strategy to establish the diagram \eqref{diag:OGLS} in odd characteristics is to first lift it to characteristic zero. To this end, we need to use relative moduli spaces of sheaves. As a preliminary remark, we point out that there is no difference between semistablity and strong semistablity for sheaves in the sense of \cite{La04} on abelian surfaces:

\begin{lemma}\label{lem:ss=strongss}
Any torsion free $H$-semistable sheaf $\cF$ on $A$ over an  algebraically closed field $k$ is strongly semistable, i.e.\ all Frobenius pull-backs $(F^m)^*\cF$, $m \geq 0$, are $H$-semistable.
\end{lemma}
\begin{proof}
The sheaf of K\"ahler differential forms $\Omega_A^1$ is trivial. Thus we have
\[
L_{\max}(\Omega^1_{A/k}) \coloneqq \lim_{m\to \infty} \frac{\mu_{\max}(F^m)^*\Omega^1_{A/k}}{p^m} =0,
\]
where the $\mu_{\max}$ means taking the maximal slope in the Harder--Narashimhan filtration. We can conclude by invoking \cite[Corollary 6.2]{La04}.
\end{proof}

Let us start the lifting procedure:

\begin{proposition}[Liftability]\label{prop:liftability}
The diagram \eqref{diag:OGLS} is valid over any algebraically closed field $k$ of characteristic $p\neq 2$. Moreover, the variety $\widetilde{M}_H(A, v)$ is liftable to characteristic zero.
\end{proposition}
\begin{proof}
First, by \cite[Proposition 6.4]{FL18}, one can lift the triple $(A, H, c_1)$ to  $(\cA, \cH, \widetilde{c}_1)$ over some mixed characteristics discrete valuation ring $W'$, which is finite over $W$ with residue field $k$.

Second, we can use Langer's construction of the relative moduli spaces of (strongly) semistable sheaves over mixed characteristic bases \cite{La04a,La04}.  For the ease of notations, we will also denote $(r,\widetilde{c}_1,s)$ by $v$. Let $P$ be the following polynomial in $\ZZ[t]$: 
\[\frac{\cH^2}{2} r t^2 + (\widetilde{c}_1 \cdot \cH) t + s.
\]
Then we can take the relative moduli space $\cM_{\cH}(\cA,v)$ over $W'$ by Langer's result, which is the fiber of the determinant morphism $\det \colon \cM_{\cH}(\cA, P) \to \Pic(\cA)$ at the section $\widetilde{c}_1 \colon W' \to \Pic(\cA)$. It uniformly corepresents the functor of $H$-semistable sheaves on $\cA$ with Mukai vector $v$ and is projective over $W'$. It implies that there is a unique proper morphism
\[
\delta \colon \cM_{\cH}(\cA,v) \times_{W'} k \to M_H(A,v).
\]
It is sufficient to show that it is an isomorphism. There is an open subscheme of $\cM_{\cH}(\cA,v)^s$ which universally corepresents the moduli functor of $H$-stable sheaves with Mukai vector $v$ on $\cA$. In other words, the formation of $\cM_{\cH}(\cA,v)^s$ commutes with any base change. In particular, 
\begin{equation}\label{eq:deltaopenlocus}
\cM_{\cH}(\cA,v)^s \times_{W'} k \cong M_H(A,v)^s.
\end{equation}
Thus $\cM_{\cH}(\cA, v)^s$ is a lifting of the coarse moduli space $M_H(A,v)^s$. From it, we can deduce that the image of $\delta$ is dense in $M_H(A,v)$ since $M_H(A,v)^s$ is open and dense in $M_H(A,v)$. However, the image of $\delta$ is closed from its properness, which implies that $\delta$ is surjecitve, hence a birational morphism. As both sides of $\delta$ are reduced and normal, by Zariski main theorem, it remains to show that $\delta$ induces an injective map on the sets of closed points. As $k$ is algebraically closed, it suffices to check the injectivity on $k$-points. In fact, if $\cH$ is $v$-generic, which can be ensured by the genericity of $H$, then any strictly semistable point $[\cF] \in \cM_{\cH}(\cA,v)$ is of form $[\cE_1 \oplus \cE_2]$, where $\cE_i$ are $W'$-flat $\cH$-semistable sheaves on $\cA$ with Mukai vector $v_0$. If $[\cE_1 \oplus \cE_2]$ and $[\cE_1' \oplus \cE_2']$ have the same image under $\delta$, then the reductions
$\cE_{1,k} \oplus \cE_{2,k}$ and $\cE_{1,k}' \oplus \cE_{2,k}' $
are $S$-equivalent. We can assume that $\cE_{1,k} \cong \cE_{1,k}'$ and $\cE_{2,k} \cong \cE_{2,k}'$. However, similarly to \eqref{eq:deltaopenlocus}, we have an isomorphism $\cM_{\cH}(\cA,v_0) \times_{W'} k \cong M_H(A,v_0)
$, hence we can conclude that $\cE_1 \cong \cE_1'$ and $\cE_2 \cong \cE_2'$. This is exactly the claimed injectivity.

Third, starting from $\mathcal{M}_\mathcal{H}(\cA, v)$, on one hand, we can perform the two blow-ups of O'Grady as in the right half of the diagram \eqref{diag:OGLS}, over $W'$, to construct
$\widehat{\mathcal{M}}_1$. On the other hand, we can also perform Lehn--Sorger's blow-up, followed by a further blow-up along $\widetilde{\Omega}$, the preimage of the (lifting of) $\Omega_v$, as in the left half of the diagram \eqref{diag:OGLS}, always over $W'$, to get $\widehat{\mathcal{M}}_2$. By the result of Lehn--Sorger \cite{LS06} recalled before, the generic fibers of $\widehat{\mathcal{M}}_1$ and $\widehat{\mathcal{M}}_2$ are isomorphic as polarized varieties. When $p\neq 2$, since their special fibers are smooth, projective (Theorem \ref{thm:OGconstruction}) and clearly non-ruled, an application of Matsusaka--Mumford's theorem \cite{MatsusakaMumford} gives us a polarized isomorphism between these special fibers.
\end{proof}

The variety $\widetilde{M}_H(A, v)$ is a 10-dimensional variety admitting a symplectic form, but it is not simply connected. We are interested in its Albanese fiber. Let $K_H(A,v)$ be the Albanese fiber of 
\begin{equation}\label{albmap}
\alb \colon M_H(A,v)\rightarrow A\times \widehat{A}.
\end{equation}
The preimage $\widetilde{K}_H(A,v)$ of $K_H(A,v)$ along O'Grady's resolution $\pi_v: \widetilde{M}_H(A,v)\to M_H(A,v)$ will be called an \emph{OG6 variety}. We show in the following theorem that the OG6 variety is a smooth symplectic variety for $p\geq 3$, and we compute its second cohomology.

\begin{theorem}
\label{prop:cohomological}
With the same assumptions as in Theorem \ref{thm:OGconstruction},
\begin{enumeratei}
    \item  if $H$ is $v$-generic, then the \emph{OG6} variety $\widetilde{K}_H(A,v)$ is a symplectic resolution of $K_H(A,v)$. It is a 6-dimensional symplectic variety and is liftable to characteristic zero;
    \item there is an isomorphism  $v^\perp \otimes K \oplus K(-1) \rightarrow H^2_{\cris}(\widetilde{K}_H(A,v))_K$ of $F$-isocrystals. Moreover, if $H^2_{\cris}(\widetilde{K}_H(A,v))$ and $H^3_{\cris}(\widetilde{K}_H(A,v))$ are $p$-torsion-free, this induces an isomorphism of $F$-crystals between $v^{\bot}\otimes W \oplus W(-1)$ and $H^2_{\cris}(\widetilde{K}_H(A,v)/W)$.
\end{enumeratei}
\end{theorem}

\begin{proof}
For (i), we have the following Cartesian diagrams
\[
\begin{tikzcd}
\widetilde{K}_H(A,v) \times A \times \widehat{A}\ar[d,"f_v \times \id_{A \times \widehat{A}}"'] \ar[r]  & \widetilde{M}_H(A,v) \ar[d,"\pi_v"]\\
K_H(A,v) \times A \times \widehat{A} \ar[r] \ar[d] & M_H(A,v) \ar[d] \\
A\times \widehat{A} \ar[r, "\times \frac{v^2}{2}"] & A \times \widehat{A}\,,
\end{tikzcd}
\]
see \cite[(4.10)]{Yo2001} and \cite[Page 3]{Yoshioka99}. If $p \nmid \frac{v^2}{2}$, then the morphism $\times \frac{v^2}{2}$ is finite and \'etale. Thus the $k$-smoothness of $\widetilde{M}_H(A,v)$ and that of $\widetilde{K}_H(A,v)$ are equivalent if $p \nmid \frac{v^2}{2}$; similarly for the existence of symplectic forms (see \cite[Propositin 6.9]{FL18}). As $v^2=8$ and $\widetilde{M}_H(A,v)$ is smooth by Theorem \ref{thm:OGconstruction}, the OG6 variety $\widetilde{K}_H(A,v)$ is smooth under the assumption $\cha(k) \neq 2$. The morphism $f_v$ is a symplectic resolution by its construction. To lift $\widetilde{K}_{\cH}(\cA,v)$, let $\widetilde{\cM}_{\cH}(\cA,v)$ be a lifting of $\widetilde{M}_H(A,v)$ over $W'$ as in Proposition \ref{prop:liftability}. Then we can take its Albanese fiber to obtain a lifting $\widetilde{\cK}_{\cH}(\cA,v)$ for $\widetilde{K}_H(A,v)$. By our construction, $\widetilde{\cK}_{\cH}(\cA,v)$ is the preimage along the $\pi_v$ of $\cK_{\cH}(\cA,v)$. Finally, to show that the \'etale fundamental group of $\widetilde{K}_{H}(A, v)$ is trivial, we use the simple connectedness of the generic fiber (a result established in characteristic zero \cite{OG6}) and the surjectivity of the specialization map for \'etale fundamental groups \cite[Expos\'e X, Corollaire 2.3]{SGA1}.

For (ii), let $K'$ be the fraction field of the base ring $W'$. By the results from characteristic zero \cite{PR13}, there is an isomorphism of $\mathbb{Z}_p$-modules at the geometric generic fiber:

\begin{equation}\label{eq:gammav}
    \gamma_v \colon v^{\bot}\otimes \mathbb{Z}_p \oplus \mathbb{Z}_p\cdot \eta \xrightarrow{\sim} H^2_{p}(\widetilde{\cK}_{\cH}(\cA,v)_{K'}),
\end{equation}
where $v^\bot$ is the orthogonal complement of $v$ in the algebraic Mukai lattice $\widetilde{H}(A)$ with respect to the Mukai pairing, and $\eta$ is the class of the exceptional divisor of Lehn--Sorger's blow-up $\widetilde{\cK}_{\cH}(\cA,v)_{K'}\to \cK_{\cH}(\cA,v)_{K'}$. The morphism $\gamma_v$ fits into the following commutative diagram
\[
\begin{tikzcd}[column sep=tiny,row sep=small]
 & v^{\bot}\otimes \mathbb{Z}_p \ar[rd,"\gamma_v|_{v^\bot\otimes \mathbb{Z}_p}","\sim"'] \ar[ld,"\gamma_v^s"']&  \\
 H^2_{p}(\cK_{\cH}(\cA,v)_{K'}^s) & & H^2_{p}(\cK_{\cH}(\cA,v)_{K'}) \ar[ll,hook',"j^*"]
\end{tikzcd}
\]
Here the $\gamma_v^s$ is induced by some semi-universal sheaf $\cE$ on $\cA \times \cK_{\cH}(\cA,v)^s$ and $j^*$ is the pull-back along the open immersion $j\colon \cK_{\cH}(\cA,v)^s_{K'} \subset \cK_{\cH}(\cA,v)_{K'}$; see \cite[Theorem 1.7]{PR13}.

We can assume that the field $K'$ is contained in $\mathbb{C}_p$. Let $G_{K'}$ be the absolute Galois group of $K'$. Since $j^*$ and $\gamma_v^s$ are both $G_{K'}$-equivariant, so is $\gamma_v$. Taking Fontaine's functor, there are isomorphisms
\[
\left(v^{\bot} \otimes_{\mathbb{Q}_p} B_{\cris}\right)^{G_{K'}} \oplus K \cdot \eta\cong \left(H^2_{p}(\widetilde{\cK}_{\cH}(\cA,v)_{K'})_{\mathbb{Q}_p} \otimes B_{\cris} \right)^{G_{K'}} \cong H^2_{\cris}(\widetilde{K}_H(A,v))_K. 
\]
The first isomorphism follows from \eqref{eq:gammav} and the fact that $\eta$ is $G_{K'}$-invariant. The second isomorphism comes from the rational $p$-adic Hodge theory. Thus it is sufficient to prove that 
\[
\left(v^{\bot} \otimes_{\mathbb{Q}_p} B_{\cris}\right)^{G_{K'}} \cong v^{\bot} \otimes K.
\]
Since $\widetilde{H}(\cA_{K'})_{\Qp} \cong (v^{\bot} \otimes \mathbb{Q}_p) \oplus \mathbb{Q}_p \cdot v$ and $(\widetilde{H}(\cA_{K'})_{\Qp} \otimes_{\mathbb{Q}_p} B_{\cris})^{G_{K'}} \cong \widetilde{H}(A)_K$, we can see
\[
\left(v^{\bot} \otimes_{\mathbb{Q}_p} B_{\cris}\right)^{G_{K'}} \oplus K(-1) \cong \widetilde{H}(A)_K \cong v^{\bot} \otimes K \oplus K \cdot v.
\]
as $F$-isocrystals. Since there is an isomorphism $K \cdot v \cong K(-1)$, we can identify $\left(v^{\bot} \otimes_{\mathbb{Q}_p} B_{\cris}\right)^{G_{K'}}$ with $v^{\bot}\otimes K$. From this, we can conclude that
\[
    v^{\bot}\otimes K \oplus K\cdot \eta \cong H^2_{\cris}(\widetilde{K}_H(A,v))_K.
\]

Recall that in \cite{Kisin06}, Kisin constructed a functor $\BK(-)$ from the category of $\ZZ_p$-lattices in crystalline representations to the category of Breuil--Kisin modules, which satisfies
\[
    \BK(-) \otimes_{W\llbracket u \rrbracket}K \simeq (- \otimes_{\QQ_p} B_{\cris})^{G_{K'}}.
\]
We can replace the Fontaine's functor by the Kisin's functor $\BK(-)$ in our previous discussion, so that we will obtain an isomorphism of Breuil--Kisin modules \footnote{Here we also need to use the integral $p$-adic comparison for abelian surface $\cA$, whose $p$-torsion-freeness is clear. We also need the fact that $\BK(\ZZ_p(-1)) = W\llbracket u \rrbracket \lbrace -1 \rbrace$, which can be found in \cite[Corollary 4.33]{BMS18}. }

\begin{equation}
    \label{eqn:overWu}
    v^{\bot} \otimes_{\ZZ_p} W\llbracket u \rrbracket \oplus W\llbracket u \rrbracket \lbrace -1 \rbrace \cong \BK(H^2_p(\widetilde{\cK}_{\cH}(\cA,v)_{K'}).
\end{equation}
Under the assumption that $H^2_{\cris}(\widetilde{K}_H(A,v)/W)$ and $H^3_{\cris}(\widetilde{K}_H(A,v)/W)$ are $p$-torsion-free, the integral $p$-adic comparison theorem \cite[Theorem 1.1 (iii)]{BMS18} allows us to obtain the desired isomorphism of $F$-crystals by tensoring both sides of \eqref{eqn:overWu} with $W$.
\end{proof}

\begin{remark}
The above proof for (ii) is indirect. The reason is that in general we can not lift $K_H(A,v)$ to $W$ (but only to a  ring $W'$ which is finite over $W$), for example, when $A$ is superspecial and $H,v$ are arbitrary. In case such lifting exists, we can apply the Berthelot--Ogus's comparison theorem \cite[Theorem 2.4]{BO78} directly to compute the $F$-(iso)crystals.
\end{remark}
\begin{remark}

If characteristic $p \geq \dim(K_H(A,v))=6$, and there is a lifting of $K_H(A,v)$ to $W_2(k)$, then the required torsion-freeness follows from the $E_1$-degeneration of the Hodge-to-de Rham spectral sequence, which is known by Deligne--Illusie's criterion \cite[Th\'eor\`eme 2.1, Corollaire 2.5]{DI87}. In fact, we can also include $p=5$ from their criterion since we only require the torsion-freeness of the second and third cohomology groups. The argument is the same as  \cite[Proposition 2.6]{Ch16}. 
\end{remark}

\section{Birational equivalences among supersingular OG6 varieties}\label{chapter:BiratSSOG6}
From now on, we focus on the OG6 varieties that are $2^{nd}$-Artin supersingular. We first prove that this condition is equivalent to the supersingularity of the abelian surface, i.e.~ $(i) \Leftrightarrow (ii)$ in Theorem \ref{thm:main}.

\begin{lemma}
\label{thm:theorem1}
Let $A$ be an abelian surface, $v_0\in \widetilde{\NS}(A)$ a primitive Mukai vector with $v_0^2=2$.
Let $v=2v_0$ and $H$ be a $v$-generic polarization. 
Then the OG6 variety $X=\widetilde{K}_{H}(A,v)$ is $2^{nd}$-Artin supersingular if and only if the abelian surface $A$ is supersingular.
\end{lemma}
\begin{proof}
Since the supersingularity of an $F$-crystal is stable under isogenies, it is sufficient to show that the Newton polygon of $H^2_{\cris}(\widetilde{K}_H(A,v))_K$ is purely of slope 1 if and only if $A$ is supersingular. In Theorem \ref{prop:cohomological}, it is shown that
\[
H^2_\cris(\widetilde{K}_H(A,v))_K \cong v^{\bot} \otimes K \oplus K(-1).
\]
As $K(-1)$ is purely of slope 1, we see that $H^2(\widetilde{K}_H(A,v))_K$ is supersingular if and only if $v^{\bot} \otimes K \subset \widetilde{H}(A)_K:=H_\cris^2(A)_K\oplus K(-1)^{\oplus 2}$ is supersingular, which is also equivalent to the condition that $H^2_\cris(A)_K$ is a supersingular crystal. This implies that the Newton polygon of $H^1_\cris(A)_K$ has to be a straight line of slope $\frac{1}{2}$, which is equivalent to the classical definition of supersingularity for $A$.
\end{proof}

In view of Lemma \ref{thm:theorem1}, we assume in the sequel that $A$ is a supersingular abelian surface over $k$. When $v$ is a primitive Mukai vector and $H$ is a $v$-generic polarization, the supersingular symplectic variety $K_H(A,v)$ is studied in our previous work \cite{FL18}. Let us recall one of the main results in \cite{FL18}, which characterizes supersingular generalized Kummer type moduli spaces up to birational equivalence.  

\begin{definition}[{\cite[Definition 3.15]{FL18}}]\label{def:quasiliftable}
Two symplectic varieties $X$ and $X'$ are said \textit{liftably birationally equivalent}, if they are both liftable over some discrete valuation ring of characteristic zero with residue field $k$, such that  their geometric generic fibers are birationally equivalent. Two symplectic varieties $X$ and $X'$ are called \textit{quasi-liftably birationally equivalent} if they can be connected by a chain of quasi-liftably birational equivalences.
\end{definition}

\begin{theorem}[{\cite[Theorem 6.12]{FL18}}, \cite{LZ20}]\label{birKum}
Let $v=(r, c_{1}, s)\in \widetilde{\NS}(A)$ be a primitive Mukai vector with $r> 0$ such that $v$ is coprime to $p$. Then there is an autoequivalence  $\Phi_{A\to A}:\rD^b(A)\rightarrow \rD^b(A)$ with $\Phi_\ast(1,0,-n-1)=v$ and $n=\frac{v^{2}}{2}-1$ which induces a quasi-liftably birational equivalence 
\begin{equation}\label{birationalKum}
M_{H}(A,v)\dashrightarrow  \Pic^0(A)\times A^{[n+1]},
\end{equation}
 and by restricting to Albanese fibers, a quasi-liftably birational equivalence
\begin{equation}\label{birationalKumtype}
K_{H}(A,v)\dashrightarrow K_{n}(A),
\end{equation}
where $K_n(A)$ denotes the $2n$-dimensional generalized Kummer variety associated with $A$.
\end{theorem}

As shown in \cite{FL18}, the same result holds for moduli spaces of stable sheaves with primitive Mukai vectors on supersingular K3 surfaces.

Our goal is to generalize Theorem \ref{birKum} to OG6 varieties.  One useful tool is the following so-called wall-crossing principle, which holds over an arbitrary base field.
\begin{proposition}
\label{wall-crossing}
Let $A$ be an abelian surface defined over a field.
Let $m \geq 1$ be a positive integer which is coprime to $p$. Let $v_0$ be a primitive Mukai vector on $A$. Let $H$ be a $mv_0$-generic ample divisor in $\NS(A)$. 
Suppose 
\[
\Phi_{A \to A'} \colon D^b(A) \to D^b(A')
\]
is a Fourier--Mukai equivalence whose cohomological transform $\phi^{H}_{A \to A'}$ sends $v_0$ to $v_0'$. Let $H'$ be an $mv_0'$-generic polarization on $A'$. Then there is a birational equivalence:
\begin{equation}
\label{gamma}
\begin{tikzcd}[column sep=small]
M_{H}(A,m{v_0}) \arrow[r,dashed,"\sim"] & M_{H'} (A',m v_0').
\end{tikzcd}
\end{equation}
Moreover, it also induces birational map between the Albanese fibers
\begin{equation}\label{eq:gammabar}
    \overline{\gamma} \colon K_H(A,mv_0) \dashrightarrow K_{H'}(A',mv_0').
\end{equation}
\end{proposition}
\begin{proof}
In \cite[Theorem 5.4.1]{MMY11}, it has been proved that for a general $H$-stable sheaf $\mathcal{E} \in M_H(A,v_0)$, if $H'$ is generic with respect to $v$, then there is an autoequivalence $\Psi_{A' \to A'}$ on $D^b(A')$ such that $\Psi_{A' \to A'}\circ \Phi_{A \to A'} (\mathcal{E})$ or its dual is $H'$-stable sheaf of Mukai vector $v_0'$ up to shift. Whence, for a general polystable sheaf 
\[
 \mathcal{E}^{\oplus m} \in M_{H}(A,mv_0)\backslash M_H^{s}(A,mv_0), 
\]
the map
\begin{equation}
\begin{split}
    \gamma \colon [\mathcal{E}^{\oplus m}] &\mapsto \Psi_{A'\to A'} \circ \Phi_{A\to A'} [\mathcal{E}^{\oplus m}]\\ &\cong [\Psi_{A' \to A'} \circ \Phi_{A \to A'} \mathcal{E}]^{\oplus m}
\end{split}
\end{equation}
is well-defined as Fourier--Mukai transforms preserve colimits. It also means that we have birational map
\[
\gamma \colon M_{H}(A,mv_0) \dashrightarrow M_{H'}(A',mv_0')
\]
as they are irreducible moduli spaces and $H,H'$ are generic. We have the following commutative diagram of isotrivial fibrations
\[
\begin{tikzcd}[column sep=small]
M_H(A,m {v_0}) \ar[r,dashed, "\gamma"] \ar[d] & M_{H'}(A',mv_0') \ar[d] \\
\widehat{A} \times A \ar[r,"\cong"] & \widehat{A'} \times A'
\end{tikzcd}
\]
Then we can see the restriction of $\gamma$ on the open-subset of $M_H(A,m v_{0})$, which is identified with some open-subset of $M_{H'}(A',mv_0')$ by $\gamma$, induces an isomorphism between two open-subsets of $K_H(A,mv_{0})$ and $K_{H'}(A',mv_0')$.
\end{proof}
\begin{remark}
The wall-crossing principle here also includes the classical wall-crossing of Gieseker moduli spaces. It means that we can assume that $H=H'$ in \eqref{gamma} and \eqref{eq:gammabar} if $H$ is both $mv_0$-generic and $mv_0'$-generic.
\end{remark}

Now we study the birational equivalences between supersingular OG6 varieties. In analogy to the primitive case in \cite{FL18}, we will show that all the OG6 varieties associated to a fixed abelian surface are quasi-liftably birational (see \cite[Definition 3.15]{FL18} for the definition). It is not hard to prove the birational equivalence, the difficult part is the quasi-liftability. This is because the auto-equivalence of supersingular abelian surfaces is not necessarily liftable. We will carefully analysis how to decompose an auto-equivalence into a series of liftable ones. 

\begin{theorem}\label{thm:Birational}
Let $v_0$ be a primitive Mukai vector with $v_0^2=2$. Set $v=2v_0$. Let $H$ be a $v$-generic polarization.
For a general polarization $H'$, there is a quasi-liftably birational equivalence \[\widetilde{K}_H(A,2v_0) \dashrightarrow \widetilde{K}_{H'}(A,(2,0,-2)).\]
In other words, for a supersingular abelian surface, OG6 varieties with different Mukai vectors are all quasi-liftably birational.
\end{theorem}
\begin{proof}
 Since $v_0$ is coprime to $p$, by Theorem \ref{birKum}, there is an autoequivalence $\Phi_{A\to A}^{\cE}:D^b(A)\rightarrow D^b(A)$ whose induced map $$\Phi^\ast:\widetilde{\NS}(A)\rightarrow \widetilde{\NS}(A)$$ sends $v_0$ to  $(1,0,-1)$. By Proposition \ref{wall-crossing}, there is a birational map
\begin{equation}\label{eq:birationalModuli}
     M_H(A,v)\dashrightarrow M_{H'}(A,(2,0,-2)).
\end{equation}
Arguing as in \cite[Proof of Theorem 6.12]{FL18}, we obtain a birational map  
\begin{equation}\label{eq:birationalK}
    \gamma \colon K_H(A,v) \dashrightarrow K_{H'}(A,(2,0,-2)).
\end{equation}
by taking the Albanese fiber. 

Proposition \ref{prop:liftability} and Lemma \ref{lemma:genericitylifting} allow us to construct the relative moduli space $\cM_{\cH}(\cA,v)$ whose generic fiber is the moduli space of semistable sheaves for some $v$-generic polarization in characteristic 0. Hence it is sufficient to choose the generic polarizations at the special fiber. We split the proof into two cases:

(1) If $\rk(v) >0$, then we can perturb the polarization $H$ within the $v$-chamber which it lies in, so that it falls in some $(2,0,-2)$-chamber. The new polarization will be denoted by $H''$. Proposition \ref{prop:10-1wallstructure} implies that $H''$ is both $v$-generic and $(2,0,-2)$-generic. Thus after doing that, we will obtain a commutative diagram of birational maps
\[
\begin{tikzcd}
K_H(A,v) \ar[d,equal] \ar[r,dashed,"\gamma"] & K_{H'}(A,(2,0,-2)) \ar[d,dashed,"\text{wall-crossing}"]\\
K_{H''}(A,v) \ar[r,dashed,"\gamma'"]& K_{H''}(A,(2,0,-2))
\end{tikzcd}
\]
By construction, the two spaces on the left are identical since the stability condition is unchanged for the Mukai vector $v$. The right vertical birational map is liftable: we first lift the triple $(A,H', H'')$ by \cite[Proposition 6.4]{FL18} to a mixed characteristics base $W'$, and then their generic fibers are birational as we can perform the wall-crossing principle for changing the lifted polarizations at the generic fibers. The birational map $\gamma'$ is also quasi-liftable since it can be constructed as follows: 

We can assume that $c_1(v_0)$ is ample as there are identifications
\[
M_{H''}(A,v) = M_{H''}(A, (2r, 2c_1 + 2n H'', 2s'))
\]
for all positive integers $n$ with suitable $s'$, since tensoring with $\mathcal{O}_A(H)$ will not change the moduli space. Moreover, we can assume that $c_1(v_0)$ is $2v_0$-generic. This is because the set of $v$-walls is locally-finite and depends only on the discriminant $\Delta$. Thus we write $v_0=(r, H''', s')$ with $H'''$ ample and $2v_0$-generic. Therefore, we
have a liftably birational map
\[
K_{H''}(A,v) \dashrightarrow K_{H'''}(A, v),
\]
by the wall-crossing of polarizations. Then we use again \cite[Proposition 6.4]{FL18} to lift the triple $(A, H''', E)$ for  $E$ the class of the fiber of some elliptic fibration of $A$, and apply the autoequivalence in Theorem \ref{birKum} to produce a quasi-liftably birational map
\[
K_{H'''}(A,v) \dashrightarrow K_{H'''}(A, (2,0,-2))
\]
as $H'''$ is $(2,0,-2)$-generic by Proposition \ref{prop:10-1wallstructure}.

(2) If $\rk(v_0)=0$, then $v_0= (0,c,s)$ such that $c$ is effective and $s \neq 0$. For a line bundle $\mathcal{L}$, the autoequivalence
\[
(-)\otimes \cL\colon D^b(A) \to D^b(A)
\]
induces a quasi-liftably birational map
\[
K_H(A,2(0,c,s)) \dashrightarrow K_H(A, 2(0,c,s+(c\cdot c_1(\mathcal{L}))).
\]
Since $(c\cdot c)=2$ and $\gcd(c \cdot H, s)=1$, we can find a divisor $m_1c+m_2H$, such that $s+(c \cdot (m_1c+m_2H))=1$.  Let $\mathcal{L}$ be a line bundle whose first Chern class is $m_1c+m_2H$, then tensoring $\mathcal{L}$ makes the degree-2 part of the Chern character equal to 1.

Therefore, we can assume that $s=1$. Then let us consider the derived equivalence $\Phi^{\cP}$ induced by the Poincar\'e bundle $\cP$ on $A \times A$ (recall that $A$ has a principal polarization). It produces a quasi-liftably birational map
\[
K_H(A,(0,2c,2)) \dashrightarrow K_H(A,(2,-2c,0)).
\]
Finally, by tensoring a line bundle whose first Chern class is $c$, we have an quasi-liftably birational map
\[
K_H(A,(2,-2c,0) ) \dashrightarrow K_H(A,(2,0,-2)). 
\]
By combining the above maps, we obtain the desired quasi-liftably birational map.
\end{proof}

\section{Unirationality and motive}\label{chapter:Proof}
In this section, we prove the main Theorem \ref{thm:main}. An important ingredient we use is the construction in \cite{MRS18} of a rational dominant degree-$2$ map from a 6-dimensional moduli space of sheaves on the Kummer K3 surface to the OG6 variety, for certain Mukai vector of \textit{torsion} sheaves. Throughout the section, $A$ is a \textit{supersingular} abelian surface defined over an algebraically closed field $k$ of characteristic $p>2$.

\subsection{Generic polarization}
Given a Mukai vector $v\in \widetilde{\NS}(A)$, we need to characterize  $v$-generic ample line bundles on $A$.  When $v$ is primitive and coprime to $p$, it is shown that every ample line bundle $H$ is $v$-generic (\cite[Lemma 6.10]{FL18}). However, this is no longer the case for non-primitive Mukai vectors. For instance,  if $A=E_1\times E_2$ is a product of two elliptic curves and $v=2(0,E_1+E_2,1)$, then by taking two stable sheaves $\cF_1$, $\cF_2$ such that $v(\cF_1)=(0,2E_1,1)$ and $v(\cF_2)=(0,2E_2,1)$, we can produce a strictly semistable sheaf $\cF_1\oplus \cF_2$ that does not satisfy the condition in Definition \ref{def:generic}. Thus the ample line bundle $\mathcal{O}_A(E_1+E_2)$ is not $v$-generic.  However, in the case where of rank-0 Mukai vector of divisibility 2, we have the following sufficient condition for the $v$-genericity.

\begin{proposition}
\label{v-generic}
Let $H$ be a principal polarization on $A$, which is not the sum of two elliptic curves. Let $v_0=(0,H ,s)$. Then $H$ is $2v_0$-generic.
\end{proposition}
\begin{proof}
Let $\cF\in M_H(A,2v_0)$. It is a purely one-dimensional sheaf with Hilbert polynomial
\[
\begin{split}
    P_H(\cF,n) = (2H\cdot H) n + 2 s
     = 4n+2s.
\end{split}
\]
Suppose $\cE \subset \cF$ is a stable factor of $\cF$. We want to show that $v(\cE)=v_0$. Denote  $v(\cE)=(0,c_1',s')$ and $v(\cF/\cE)=(0,c_1'',s'')$. Then we have 
$c_1'+c_1''=2H$ and $s'+s''=2s$. Since $\cE$ and $\cF/\cE$  have the same reduced Hilbert polynomials as that of $\cF$, namely $n+\frac{1}{2}s$, we have
\begin{equation}
    (c_1'\cdot H)s =2s', \quad (c_1''\cdot H)s=2s''.
\end{equation}

By definition, $\cE$ is also purely one-dimensional. We claim that $\cE$ is a stable sheaf such that $\supp_{\det}(\cE)$\footnote{This is the determinant support, see \cite[Section 3.1]{ASF2015} for the definition.} is an irreducible curve $C$ in $A$, and the  restriction of $\cE$ on $C$ is locally-free. First, to see the irreducibility of $C$, suppose otherwise that $C_1 + C_2$ is a decomposition of $C$ and $j_i \colon C_i \to C$ is the embedding respectively. Then we can see that the quotient sheaf $$\cE|_{C} \to j_{i,*}j^*_i\cE|_{C}$$ has the same reduced Hilbert polynomial as $\mathcal{E}$, a contradiction to the stability of $\cE|_C$. Hence $C$ is irreducible. As $\cE|_C$ is stable, hence simple, we have $\dim_k H^0(C, \cE|_{C}\otimes \cE|_{C}^{\vee})=1$, which yields that
\begin{equation}
    1- \dim H^1(C,\cE|_{C} \otimes \cE|_{C}^{\vee}) = \chi(\cE , \cE) = -(c_1')^2
\end{equation}
Thus $(c_1')^2 \geq -1$. As the Neron-Severi lattice of $A$ is even, $(c_1')^2\geq0$. 

On the other hand, we can assume that $\cE$ and $\cF/\cE$ not isomorphic (otherwise $v(\cE)=v(\cF/\cE)=v_0$ and we are done), but they are stable with the same reduced Hilbert polynomial, thus $\Hom(\cE,\cF/\cE)=0$. It implies that
\begin{equation}
    \dim\Ext^1(\cE,\cF/\cE) = -\chi(\cE,\cF/\cE) = (c_1' \cdot c_1'')
\end{equation}
Hence $(c_1'\cdot c_1'')\geq 0$. Then by the equations $c_1'+c_1''=2H$, we can conclude that $(c_1'\cdot H) \geq 0$ and $(c_1''\cdot H) \geq 0$.  Therefore we must have $s'=s''=s$, $(c_1'\cdot H)=(c_1''\cdot H)=2$ and $0 \leq (c_1')^2=(c_1'')^2\leq 2$.

If $(c_1')^2=0$, then $c_1'$ is the class of an elliptic curve of degree $2$ on $A$,  contradicting to the assumption. Thus $(c_1')^2>0$. It forces that $(c_1')^2=(c_1'')^2=2$. Now the Hodge index theorem applied to the divisors $H-c_1'$ and $H-c_1''$, gives us that $c_1'=c_1''=H$. In conclusion, all stable factors of $\cF$ must have Mukai vector $v_0$.
\end{proof}

\subsection{Mongardi--Rapagnetta--Sacc\`a double cover}
\begin{proposition}
\label{prop1}
Let $v_0=(0,c_1,s)$ be a primitive Mukai vector with $v_0^2=2$ and let $H$ be a $2v_0$-generic polarization on $A$. If $c_1$ is not the sum of the classes of two elliptic curves, then there is a dominant degree-2 rational map
\begin{equation}
M_{H'} (\mathrm{Km}(A), \tilde{v})\dashrightarrow K_{H}(A,2v_0)
\end{equation}
from the moduli space of $H'$-stable sheaves on the Kummer surface $\km(A)$ for some Mukai vector $\tilde{v}$ and some $\tilde{v}$-generic ample line bundle $H'$.
\end{proposition}
\begin{proof}
This is explained in the introduction of \cite{MRS18}. We include a proof for readers' convenience.
We may assume that $c_1=c_1(L)$ is the class of a principal polarization $L$ of $A$ and the general members in $|L|$ are smooth curves of genus 2. 
Let $$\varphi_{|2L|} \colon A \rightarrow  |2L|^\vee =\mathbb{P}^3_k$$ 
be the morphism associated with the linear system $|2L|$. It is well known that the image of $\varphi_{|2L|}$ in $\PP^3_k$ is a singular quartic surface, which is isomorphic to $\km_s(A)\coloneqq A/\pm $. The Kummer K3 surface $\km(A)$ is the minimal resolution of $\km_s(A)$.  We have the following cartesian diagram:
\begin{equation*}
    \begin{tikzcd}
    \bl_{A[2]}A  \ar[r] \ar[d]& \km{A} \ar[d, "\pi"]\\
    A \ar[r, "\varphi"]& \km_s{A}=A/\pm
    \end{tikzcd}
\end{equation*}

By Lemma \ref{v-generic}, $L$ is  $2v_0$-generic. We treat first the case $\widetilde{K}_L(A, 2v_0)$. Let $D$ be pull-back of a hyperplane section of $\km_s(A)$ along the resolution $\pi$. Suppose $H'$ is an ample divisor in $\NS(\km(A))$ which is generic with respect to the Mukai vector $\widetilde{v}=(0,D,s)$. Then we have the following diagram, where the vertical arrows are support morphisms:
\begin{equation}
\begin{tikzcd}
M_{H'}(\km(A),\widetilde{v})\ar[d]  & M_{L}(A,2v_0) \ar[d] \\
\vert D\vert \ar[r,"\cong"] & \vert 2L\vert
\end{tikzcd}
\end{equation}
Now, we claim that the map $[\mathcal{E}] \mapsto [\varphi^* \pi_{*} \mathcal{E}]$ defines a rational map
\begin{equation}\label{dominant-mod}
M_{H'}(\km(A),\widetilde{v}) \dashrightarrow M_{L}(A,2v_0).
\end{equation}In fact, we can choose a general point $[\mathcal{E}] \in M_{H'}(\km(A),\widetilde{v})$, which corresponds to a purely one-dimensional $H'$-stable sheaf $\mathcal{E}$ that supports on a smooth curve $C \in |D|$. Thus it can be viewed as a locally-free sheaf on $C \in |D| \cong |2L|$. Its pull-back $\varphi^* \mathcal{E}$ on $A$ is a $L$-stable sheaf  with Mukai vector $(0, 2c_1,2s)$ as the restriction of $\varphi$ on $C$ is a degree-2 \'etale cover.

Next, we need to show that the rational map \eqref{dominant-mod} factors through $K_L(A,v)$ and it is dominant. 
The generic image of the rational map \eqref{dominant-mod}  lies in the zero fiber $K_{L}(A,2v)$, because the Albanese invariant of $c_2(\pi^* \mathcal{E})= \pi^* c_2(\mathcal{E})$, which is pulled back from the quotient $A/\pm$, is zero. Hence \eqref{dominant-mod} gives a rational map
\begin{equation}
\label{eq:dominantmap}
M_{H'}(\km(A),\widetilde{v}) \dashrightarrow \widetilde{K}_{L}(A,2v_0),
\end{equation}
which is dominant, as both varieties are proper, irreducible, and 6-dimensional. 

It remains to prove the case of some other $2v_0$-generic polarization $H$. This follows from the birational equivalence between $M_H(A,2v_0)$ and $M_L(A,2v_0)$ proved in Proposition \ref{wall-crossing}.
\end{proof}

\subsection{Unirationality}\label{subsec:Unirat}
We prove the unirationality for supersingular OG6 varieties.

\begin{proof}[Proof of Theorem \ref{thm:main} $(i)\Rightarrow (iii)$]
By Theorem \ref{thm:Birational}, all OG6 varieties for different choices of Mukai vectors are birationally equivalent, hence it suffices to show the unirationality of $\widetilde{K}_H(A,\bar{v})$ for a specific Mukai vector $\bar{v}$. Here is the construction of such a $\bar{v}$. According to the a result of Ibukiyama, Katsura and Oort \cite{IKO86} (not explicitly stated, but see \cite[Theorem 5]{Ka14}),  every supersingular abelian surface $A$ contains a smooth curve $C$ of genus $2$ as Th\^eta divisor. We take $\bar{v}=(0, 2\Theta, 2)$, where $\Theta$ is the class of $C$. By Proposition \ref{prop1}, for a $\bar{v}$-generic polarization $H$, the resulting OG6 variety $\widetilde{K}_H(A,\bar{v})$ admits a dominant map from $M_{H'} (\mathrm{Km}(A),v')$. However, by \cite[Theorem 1.3]{FL18}, $M_{H'} (\mathrm{Km}(A),v')$ is birational to $\mathrm{Km}(A)^{[3]}$, the Hilbert cube of the Kummer K3 surface  $\mathrm{Km}(A)$. One can conclude by using the unirationality of $\mathrm{Km}(A)$ established by Shioda \cite{Shioda77}. Note that we did not use liftability results in this proof.
\end{proof}

\subsection{Motive of supersingular OG6 varieties}
The following theorem computes the rational Chow motive of supersingular OG6 varieties.
\begin{theorem}\label{thm:TateMotive}
Let $k$ be an algebraically closed field of characteristic $p\geq 3$. Let $A$ be a supersingular abelian surface over $k$ and $v_0$ a primitive Mukai vector with $v_0^2=2$. Set $v=2v_0$. For any $v_0$-generic polarization $H$, the motive of the supersingular OG6 variety $\widetilde{K}_H(A,v)$ is of Tate type:
\begin{equation}\label{eq:motiveOG6}
    \h(\widetilde{K}_H(A,v))=\bigoplus_{i=0}^{6} \1(-i)^{\oplus b_{2i}},
\end{equation}
where $b_{2i}$ is the $2i$-th Betti number of OG6 varieties, as is computed in \cite{MRS18}.
\end{theorem}
\begin{proof}
Thanks to \cite[\S 3.6]{FL18}, Chow motive is invariant under quasi-liftably birational equivalence for symplectic varieties. By Proposition \ref{thm:Birational}, we only need to treat the special case where $v_0=(0, \Theta, 1)$ with $\Theta=[C]$ where $C$ is the smooth genus 2 curve in $A$ as in \S \ref{subsec:Unirat}.
Following \cite{MRS18}, let $D\in \Pic(\mathrm{Km}(A))$ be the pull-back of the hyperplane section class from the singular Kummer surface $\mathrm{Km}_s(A)$, and let $\underline{Y}:=M_{H'} (\mathrm{Km}(A), (0, D, 1))$ be the 6-dimensional moduli space of stable sheaves on $\mathrm{Km}(A)$ with Mukai vector $(0, D, 1)$ with respect to some generic polarization $H'$. 

Now we use \cite[Propsition 6.4]{FL18} to lift the abelian variety $A$ together with the divisors $\Theta$ and $H'$, over a discrete valuation ring $W'$ which is finite over $W$ and with residue field $k$. Denote by $\mathcal{A}$ the lifted abelian scheme. Using the construction of Langer \cite{La04a, La04} of relative moduli spaces of sheaves on $\mathcal{A}$ and $\mathrm{Km}(\mathcal{A})$, together with our discussion in \S \ref{sec:OGLS-Lifting} on the resolution (Theorem \ref{prop:cohomological}), we get the lifting of $\widetilde{K}:=\widetilde{K}_H(A, (0, 2\Theta, 2))$, denoted by $\widetilde{\mathcal{K}}$, and the lifting of $M_{H'}(\mathrm{Km}(A), (0, D, 1))$, denoted by $\underline{\mathcal{Y}}$, as well as a rational map 
\[\xymatrix{\underline{\mathcal{Y}} \ar@{-->}[r]& \widetilde{\mathcal{K}},}\]
which is the lifting of the degree 2 rational map constructed in Proposition \ref{prop1}.

On the geometric generic fiber, according to the result in characteristic zero \cite[\S 5]{MRS18}, the rational map $\underline{\mathcal{Y}}_{\bar{\eta}}\dashrightarrow \widetilde{\mathcal{K}}_{\bar{\eta}}$ can be resolved as follows:
\begin{equation}\label{diag:Resolve}
    \xymatrix{
    \widehat{\mathcal{Y}}_{\bar\eta} \ar[d]_\beta \ar[dr]^{\widehat{\varepsilon}} &\\
    \overline{\mathcal{Y}}_{\bar\eta} \ar[d]_h & \widehat{\mathcal{K}}_{\bar\eta} \ar[d]^\rho\\
     \underline{\mathcal{Y}}_{\bar\eta} \ar@{-->}[r] & \widetilde{\mathcal{K}}_{\bar\eta}
    }
\end{equation}
where $h$ is the blow-up along the disjoint union of 256 copies of $\mathbb{P}^3_{\bar\eta}$, $\beta$ is the blow-up along $\overline{\Delta}\simeq \bl_{\mathcal{A}_{\bar\eta}\times \mathcal{A}_{\bar\eta}^\vee[2]}(\mathcal{A}_{\bar\eta}\times \mathcal{A}_{\bar\eta}^\vee)/{\pm 1}$, $\rho$ is the blow-up along the disjoint union of 256 three-dimensional quadrics, and $\widehat{\varepsilon}$ is a ramified double cover. As the notation indicates, all the blow-ups can be performed over $W'$ and the diagram \eqref{diag:Resolve} is indeed the base change to the geometric generic fiber of the analoguous diagram over $W'$, except that a priori $\widehat{\epsilon}$ is only a rational map over $W'$.

Using the surjectivity and the projectivity of the morphism $\widehat{\mathcal{Y}}_{\bar\eta}\to \widetilde{\mathcal{K}}_{\bar\eta}$, we obtain that there is a splitting injection from the Chow motive of $\widetilde{\mathcal{K}}_{\bar\eta}$ into the Chow motive of $\widehat{\mathcal{Y}}_{\bar\eta}$. By specialization, the Chow motive of the special fiber $\widetilde{K}$ is a direct summand of the Chow motive of the special fiber of $\widehat{\mathcal{Y}}$, denoted by $\widehat{Y}$. Therefore, to prove that $\h(\widetilde{K})$ is of Tate type, it suffices to show the same for $\h(\widehat{Y})$.

However, by construction, $\widehat{Y}$ is the obtained from $\underline{Y}$ by first blow-up along 256 copies of $\mathbb{P}_k^3$, resulting a space $\overline{Y}$, and then blow-up $\overline{Y}$ along a center $\overline{\Delta}\simeq \bl_{A\times A^\vee[2]}(A\times A^\vee)/{\pm 1}$.

By \cite[Theorem 1.3(iii)]{FL18}, the Chow motive of $\underline{Y}$ is of Tate type. Since the Chow motive of $\mathbb{P}^3$ is of Tate type, the blow-up formula implies that the Chow motive of $\overline{Y}$ is of Tate type. Observe that 
\[\h(\overline{\Delta})\simeq \h(A\times A^\vee)^{\operatorname{inv}}\oplus \left(\1(-1)\oplus \1(-2)\oplus \1(-3)\right)^{\oplus 256}.\]
However, by the computation of the motive of supersingular abelian varieties \cite[Theorem 2.9]{FL18}, $\h(A\times A^\vee)^{\operatorname{inv}}$ is of Tate type. Another application of the blow-up formula yields that the Chow motive of $\widehat{Y}$ is of Tate type. Hence $\h(\widetilde{K})$ is also of Tate type. The explicit formula \eqref{eq:motiveOG6} can be obtained by looking at the cohomological realization.
\end{proof}

\subsection{Proof of Theorem \ref{thm:main}}

$(i)\Leftrightarrow (ii)$ is Lemma \ref{thm:theorem1}.\\ $(i)\Rightarrow (iii)$ is proved in \S \ref{subsec:Unirat}.\\
$(iii)\Rightarrow (ii)$: unirationality clearly implies rational chain connectedness. We use \cite[Theorem 1.2]{GJ17} (see \cite[Theorem 3.10]{FL18}) to conclude the $2^{nd}$-Shioda supersingularity, which implies the $2^{nd}$-Artin supersingularity. \\
$(i)\Rightarrow (iv)$ is contained in Theorem \ref{thm:TateMotive}.\\
$(iv)\Rightarrow (v)$ easily follows by applying the cohomological realizations.\\
$(v) \Rightarrow (ii)$ is clear since $2^{nd}$-Shioda supersingularity implies $2^{nd}$-Artin supersingularity.

It only remains to verify that $\dim H^{2,0}(\widetilde{K}_H(A,v))=1$, under the assumption that $A$ is supersingular. Since $h^{2,0}$ is a birational invariant, Theorem \ref{thm:Birational} implies that it suffices to show the case when $v=2(0,\Theta,1)$ with $\Theta\in \NS(A)$ the class of the principal polarization. From the proof in Theorem \ref{thm:TateMotive}, we have the Mongardi--Rapagnetta--Sacc\`a double cover 
$$\psi:\underline{Y}\dashrightarrow \widetilde{K}_H(A,v),$$
where $\underline{Y}$ is a moduli space of stable sheaves on the Kummer surface $\operatorname{Km}(A)$.
It is easy to see that the pullback map $$\psi^\ast: H^0(\widetilde{K}_H(A,v), \Omega^2_{\widetilde{K}_H(A,v) })\rightarrow H^0(\underline{Y},\Omega_{\underline{Y}}^2)$$ is injective when $p>2$. Hence $h^{2,0}(\widetilde{K}_H(A,v))\leq h^{2,0}(\underline{Y})$. 
By \cite[Theorem 1.3 (ii)]{FL18}, $\underline{Y}$ is an irreducible symplectic variety (in particular, $h^{2,0}(\underline{Y})=1$). It follows that $h^{2,0}(\widetilde{K}_H(A,v))=1$ in this case.   

\bibliographystyle {abbrv}
\bibliography{main}
\end{document}